\documentclass{amsart}
\usepackage[square]{natbib}
\newtheorem{thm}{Theorem}
\newtheorem{lem}[thm]{Lemma}
\newtheorem{cor}[thm]{Corollary}

\theoremstyle{definition}

\newtheorem{ex}[thm]{Example}

\providecommand{\abs}[1]{\lvert#1\rvert}
\providecommand{\Abs}[1]{\Bigl\lvert#1\Bigr\rvert}

\begin{document}
\title[FTAP]{Two versions of the fundamental\\theorem of asset pricing}

\author{Patrizia Berti}
\address{Patrizia Berti, Dipartimento di Matematica Pura ed Applicata ``G. Vitali'', Universita' di Modena e Reggio-Emilia, via Campi 213/B, 41100 Modena, Italy}
\email{patrizia.berti@unimore.it}

\author{Luca Pratelli}
\address{Luca Pratelli, Accademia Navale, viale Italia 72, 57100 Livorno,
Italy} \email{pratel@mail.dm.unipi.it}

\author{Pietro Rigo}
\address{Pietro Rigo (corresponding author), Dipartimento di Matematica ``F. Casorati'', Universita' di Pavia, via Ferrata 1, 27100 Pavia, Italy}
\email{pietro.rigo@unipv.it}

\keywords{Arbitrage, Convex cone, Equivalent
martingale measure, Equivalent probability measure with given marginals, Finitely additive probability,
Fundamental theorem of asset pricing}

\subjclass[2010]{60A05, 60A10, 28C05, 91B25, 91G10}

\begin{abstract}
Let $L$ be a convex cone of real random variables on the probability space $(\Omega,\mathcal{A},P_0)$. The existence of a probability $P$ on $\mathcal{A}$ such that
\begin{equation*}
P\sim P_0,\quad E_P\abs{X}<\infty\,\text{ and }\,E_P(X)\leq 0\,\text{ for all }X\in L
\end{equation*}
is investigated. Two results are provided. In the first, $P$ is a finitely additive probability, while $P$ is $\sigma$-additive in the second. If $L$ is a linear space then $-X\in L$ whenever $X\in L$, so that $E_P(X)\leq 0$ turns into $E_P(X)=0$. Hence, the results apply to various significant frameworks, including equivalent martingale measures and equivalent probability measures with given marginals.
\end{abstract}

\maketitle

\section{Introduction}\label{mot}

Throughout, $(\Omega,\mathcal{A},P_0)$ is a probability space and $L$ a convex cone of real random variables on $(\Omega,\mathcal{A},P_0)$. We focus on those probabilities $P$ on $\mathcal{A}$ such that
\begin{equation}\label{bgf57}
P\sim P_0,\quad E_P\abs{X}<\infty\,\text{ and }\,E_P(X)\leq 0\,\text{ for all }X\in L.
\end{equation}
Our main concern is the existence of one such $P$. Essentially two results are provided. In the first, $P$ is a finitely additive probability, while $P$ is $\sigma$-additive in the second. The reference probability $P_0$ is $\sigma$-additive.

In economic applications, for instance, $L$ could be a collection of random variables dominated by stochastic integrals of the type $\int_0^1 H\,dS$, where the semimartingale $S$ describes the stock-price process, and $H$ is a predictable $S$-integrable process ranging in some class of admissible trading strategies; see \cite{ROKSCA}.

However, even if our results apply to any convex cone $L$, this paper has been mostly written having a linear space in mind. In fact, if $L$ is a linear space, since $-X\in L$ whenever $X\in L$, condition \eqref{bgf57} yields
\begin{equation*}
E_P(X)=0\quad\text{for all }X\in L.
\end{equation*}
Therefore, the addressed problem can be motivated as follows.

Let $S=(S_t:t\in T)$ be a real process on $(\Omega,\mathcal{A},P_0)$ indexed by $T\subset\mathbb{R}$. Suppose $S$ is adapted to a filtration $\mathcal{G}=(\mathcal{G}_t:t\in T)$ and $S_{t_0}$ is a constant random variable for some $t_0\in T$. A classical problem in mathematical finance is the existence of an {\it equivalent martingale measure}, that is, a $\sigma$-additive probability $P$ on $\mathcal{A}$ such that $P\sim P_0$ and $S$ is a $\mathcal{G}$-martingale under $P$. But, with a suitable choice of the linear space $L$, an equivalent martingale measure is exactly a $\sigma$-additive solution $P$ of \eqref{bgf57}. It suffices to take $L$ as the linear space generated by the random variables
\begin{equation*}
I_A\,(S_u-S_t)\quad\text{for all }u,\,t\in T\text{ with }u>t\text{ and }A\in\mathcal{G}_t.
\end{equation*}
Note also that, if $L$ is taken to be the convex cone generated by such random variables, a $\sigma$-additive solution $P$ of \eqref{bgf57} is an {\it equivalent super-martingale measure}.

Equivalent martingale measures are usually requested to be $\sigma$-additive, but their economic interpretation is preserved if they are only finitely additive. Thus, to look for finitely additive equivalent martingale measures seems to be reasonable. We refer to \cite{BPR1}-\cite{BPR2} and the beginning of Section \ref{st3} for a discussion on this point.

Equivalent martingale measures (both $\sigma$-additive and finitely additive) are the obvious motivation for our problem, and this explains the title of this paper. But they are not the only motivation. There are other issues which can be reduced to the existence of a probability $P$ satisfying \eqref{bgf57} for a suitable linear space $L$ (possibly without requesting $P\sim P_0$). One example are equivalent probability measures with given marginals; see Example \ref{u6}. Another example is compatibility of conditional distributions; see e.g. \cite{BDR}. A last example is de Finetti's coherence principle and related topics; see \cite{BR96} and references therein.

This paper consists of two results (Theorems \ref{t1} and \ref{p3b}) some corollaries and a long final section of examples.

In Theorem \ref{t1}, under the assumption that each $X\in L$ is bounded, the existence of a finitely additive probability $P$ satisfying \eqref{bgf57} is given various characterizations. As an example, one such $P$ exists if and only if
\begin{equation*}
\bigl\{P_0(X\in\cdot):\,X\in L,\,\,X\geq -1\text{ a.s.}\bigr\}
\end{equation*}
is a tight collection of probability laws on the real line. Furthermore, under some conditions, Theorem \ref{t1} also applies when the elements of $L$ are not bounded; see Corollary \ref{fol7f3b}.

Theorem \ref{p3b} is our main result. No assumption on the convex cone $L$ is required. A $\sigma$-additive probability $P$ satisfying \eqref{bgf57} is shown to exist if and only if
\begin{gather*}
E_Q\abs{X}<\infty\quad\text{and}\quad E_Q(X)\leq k\,E_Q(X^-)
\end{gather*}
for all $X\in L$, some constant $k\geq 0$ and some $\sigma$-additive probability $Q$ such that $Q\sim P_0$. Moreover, when applied with $Q=P_0$, the above condition amounts to the existence of a $\sigma$-additive probability $P$ satisfying \eqref{bgf57} and $r\,P_0\leq P\leq s\,P_0$ for some constants $0<r\leq s$.

Some applications of Theorems \ref{t1} and \ref{p3b} are given in Section \ref{sg3}.

\section{Notation}\label{not6}

In the sequel, as in Section \ref{mot}, $L$ is a convex cone of real random variables on the fixed probability space $(\Omega,\mathcal{A},P_0)$. Thus,
\begin{equation*}
\sum_{j=1}^n\lambda_j\,X_j\in L\quad\text{for all }n\geq 1,\,\lambda_1,\ldots,\lambda_n\geq 0\text{ and }X_1,\ldots,X_n\in L.
\end{equation*}

We let $\mathbb{P}$ denote the set of finitely additive probabilities on $\mathcal{A}$ and $\mathbb{P}_0$ the subset of those $P\in\mathbb{P}$ which are $\sigma$-additive. Recall that $P_0\in\mathbb{P}_0$.

Sometimes, $L$ is identified with a subset of $L_p$ for some $0\leq p\leq\infty$, where
\begin{equation*}
L_p=L_p(\Omega,\mathcal{A},P_0).
\end{equation*}
In particular, $L$ can be regarded as a subset of $L_\infty$ if each $X\in L$ is bounded.

For every real random variable $X$, we let
\begin{gather*}
\text{ess sup}(X)=\inf\{x\in\mathbb{R}:P_0(X>x)=0\}\quad\text{where }\inf\emptyset=\infty.
\end{gather*}

Given $P,\,T\in\mathbb{P}$, we write $P\ll T$ to mean that $P(A)=0$ whenever $A\in\mathcal{A}$ and $T(A)=0$. Also, $P\sim T$ stands for $P\ll T$ and $T\ll P$.

Let $P\in\mathbb{P}$ and $X$ a real random variable. We write
\begin{equation*}
E_P(X)=\int XdP
\end{equation*}
whenever $X$ is $P$-integrable. Every bounded random variable is $P$-integrable. If $X$ is unbounded but $X\geq 0$, then $X$ is $P$-integrable if and only if $\inf_nP(X>n)=0$ and $\sup_n\int X\,I_{\{X\leq n\}}\,dP<\infty$. In this case,
\begin{equation*}
\int X\,dP=\sup_n\int X\,I_{\{X\leq n\}}\,dP.
\end{equation*}
An arbitrary real random variable $X$ is $P$-integrable if and only if $X^+$ and $X^-$ are both $P$-integrable, and in this case $\int XdP=\int X^+dP-\int X^-dP$.

In the sequel, a finitely additive solution $P$ of \eqref{bgf57} is said to be an {\em equivalent super-martingale finitely additive probability} (ESFA). We let $\mathbb{S}$ denote the (possibly empty) set of ESFA's. Thus, $P\in\mathbb{S}$ if and only if
\begin{equation*}
P\in\mathbb{P},\quad P\sim P_0,\quad X\text{ is }P\text{-integrable}\quad\text{and}\quad E_P(X)\leq 0\,\text{ for each }X\in L.
\end{equation*}
Similarly, a $\sigma$-additive solution $P$ of \eqref{bgf57} is an {\em equivalent super-martingale measure} (ESM). That is, $P$ is an ESM if and only if $P\in\mathbb{P}_0\cap\mathbb{S}$. Recall that, if $L$ is a linear space and $P$ is an ESFA or an ESM, then $E_P(X)=0$ for all $X\in L$.

Finally, it is convenient to recall the classical no-arbitrage condition
\begin{equation}\label{nac}
L\cap L_0^+=\{0\}\quad\text{or equivalently}\quad (L-L_0^+)\cap L_0^+=\{0\}.\tag{NA}
\end{equation}

\section{Equivalent super-martingale finitely additive probabilities}\label{st3}

In \cite{BPR1}-\cite{BPR2}, ESFA's are defended via the following arguments.

\vspace{0.2cm}

\begin{itemize}

\item The finitely additive probability theory is well founded and developed, even if not prevailing. Among its supporters, we mention B. de Finetti, L.J. Savage and L.E. Dubins.

\item It may be that ESFA's are available while ESM's fail to exist.

\item In option pricing, when $L$ is a linear space, ESFA's give arbitrage-free prices just as ESM's. More generally, the economic motivations of martingale probabilities do not depend on whether they are $\sigma$-additive or not. See e.g. \cite[Chapter 1]{DS05}.

\item Approximations. Each ESFA $P$ can be written as $P=\delta\,P_1+(1-\delta)\,Q$, where $\delta\in [0,1)$, $P_1\in\mathbb{P}$, $Q\in\mathbb{P}_0$ and $Q\sim P_0$. Thus, when ESM's fail to exist, one might be content with an ESFA $P$ with $\delta$ small enough. An extreme situation of this type is exhibited in Example \ref{s4r}.

\end{itemize}

\vspace{0.2cm}

This section deals with ESFA's. Two distinct situations (the members of $L$ are, or are not, bounded) are handled separately.

\subsection{The bounded case}\label{y7j9}

In this Subsection, $L$ is a convex cone of real {\it bounded} random variables. Hence, the elements of $L$ are $P$-integrable for any $P\in\mathbb{P}$.

We aim to prove a sort of {\it portmanteau} theorem, that is, a result which collects various characterizations for the existence of ESFA's. To this end, the following technical lemma is needed.

\begin{lem}\label{kumon} Let $C$ be a convex class of real bounded random variables, $\phi:C\rightarrow\mathbb{R}$ a linear map, and $\mathcal{E}\subset\mathcal{A}$ a
collection of nonempty events such that $A\cap B\in\mathcal{E}$
whenever $A,\,B\in\mathcal{E}$. There is $P\in\mathbb{P}$ satisfying
\begin{equation*}
\phi(X)\leq E_P(X)\quad\text{and}\quad P(A)=1\quad\text{for all }X\in
C\text{ and }A\in\mathcal{E}
\end{equation*}
if and only if
\begin{equation*}
\sup_AX\geq\phi(X)\quad\text{for all }X\in C\text{ and
}A\in\mathcal{E}.
\end{equation*}
\end{lem}

\begin{proof}
This is basically \cite[Lemma 2]{BPR1} and so we just give a sketch of the proof. The ``only if'' part is trivial. Suppose $\sup_AX\geq\phi(X)$ for all $A\in\mathcal{E}$ and $X\in C$. Fix $A\in\mathcal{E}$ and define $C_A=\{X|A-\phi(X):X\in C\}$, where $X|A$ denotes the restriction of $X$ on $A$. Then, $C_A$ is a convex class of bounded functions on $A$ and $\sup_A Z\geq 0$ for each $Z\in C_A$. By \cite[Lemma 1]{HS}, there is a finitely additive probability $T$ on the power set of $A$ such that $E_T(Z)\geq 0$ for each $Z\in C_A$. Define
\begin{equation*}
P_A(B)=T(A\cap B)\quad\text{for }B\in\mathcal{A}.
\end{equation*}
Then, $P_A\in\mathbb{P}$, $P_A(A)=1$ and $E_{P_A}(X)=E_T(X|A)\geq\phi(X)$ for each $X\in C$. Next, let $\mathcal{Z}$ be the set of all functions from $\mathcal{A}$ into $[0,1]$, equipped with the product topology, and let
\begin{equation*}
F_A=\bigl\{P\in\mathbb{P}:P(A)=1\text{ and }E_P(X)\geq\phi(X)\text{ for all }X\in C\bigr\}\quad\text{for }A\in\mathcal{E}.
\end{equation*}
Then, $\mathcal{Z}$ is compact and $\{F_A:A\in\mathcal{E}\}$ is a collection of closed sets satisfying the finite intersection property. Hence, $\bigcap_{A\in\mathcal{E}}F_A\neq\emptyset$, and this concludes the proof.
\end{proof}

We next state the portmanteau theorem for ESFA's. Conditions (a)-(b) are already known while conditions (c)-(d) are new. See \cite[Theorem 2]{CAS}, \cite[Theorem 2.1]{ROK} for (a) and \cite[Theorem 3]{BPR1}, \cite[Corollary 1]{ROKSCA} for (b); see also \cite{ROKHL}. Recall that $\mathbb{S}$ denotes the (possibly empty) set of ESFA's and define
\begin{equation*}
\mathcal{Q}=\{Q\in\mathbb{P}_0:Q\sim P_0\}.
\end{equation*}

\begin{thm}\label{t1} Let $L$ be a convex cone of real bounded random variables. Each of the following conditions is equivalent to $\mathbb{S}\neq\emptyset$.

\begin{itemize}

\item[(a)] $\overline{L-L_\infty^+}\,\cap L_\infty^+=\{0\}$, with the closure in the norm-topology of $L_\infty$;

\vspace{0.2cm}

\item[(b)] There are $Q\in\mathcal{Q}$ and a constant $k\geq 0$ such that
\begin{equation*}
E_Q(X)\leq k\,\text{ess sup}(-X)\quad\text{for each }X\in L;
\end{equation*}

\item[(c)] There are events $A_n\in\mathcal{A}$ and constants $k_n\geq 0$, $n\geq 1$, such that
\begin{gather*}
\lim_nP_0(A_n)=1\quad\text{and}
\\E_{P_0}\bigl\{X\,I_{A_n}\bigr\}\leq k_n\,\text{ess sup}(-X)\quad\text{for all }n\geq 1\text{ and }X\in L;
\end{gather*}

\item[(d)] $\bigl\{P_0(X\in\cdot):\,X\in L,\,\,X\geq -1\text{ a.s.}\bigr\}$ is a tight collection of probability laws.

\end{itemize}

\vspace{0.2cm}

\noindent Moreover, under condition (b), an ESFA is
\begin{equation*}
P=\frac{Q+k\,P_1}{1+k}\quad\text{ for a suitable }P_1\in\mathbb{P}.
\end{equation*}

\end{thm}

\begin{proof} First note that each of conditions (b)-(c)-(d) implies \eqref{nac}, which in turn implies
\begin{equation*}
\text{ess sup}(X^-)=\text{ess sup}(-X)>0\quad\text{whenever }X\in L\text{ and }P_0(X\neq 0)>0.
\end{equation*}

{\bf (b) $\Rightarrow$ (c).} Suppose (b) holds. Define $k_n=n\,(k+1)$ and $A_n=\{n\,f\geq 1\}$, where $f$ is a density of $Q$ with respect to $P_0$. Since $f>0$ a.s., then $P_0(A_n)\rightarrow 1$. Further, condition (b) yields
\begin{gather*}
E_{P_0}\bigl\{X\,I_{A_n}\bigr\}\leq E_{P_0}\bigl\{X^+\,I_{A_n}\bigr\}=E_Q\bigl\{X^+\,(1/f)\,I_{A_n}\bigr\}\leq n\,E_Q(X^+)
\\=n\,\bigl\{E_Q(X)+E_Q(X^-)\bigr\}\leq n\,\bigl\{k\,\text{ess sup}(-X)+\text{ess sup}(X^-)\bigr\}
\\=k_n\,\text{ess sup}(-X)\quad\text{for all }n\geq 1\text{ and }X\in L.
\end{gather*}

{\bf (c) $\Rightarrow$ (d).} Suppose (c) holds and define $D=\{X\in L:X\geq -1\text{ a.s.}\}$. If there is a subsequence $n_j$ such that $k_{n_j}\leq 1$ for all $j$, taking the limit as $j\rightarrow\infty$  condition (c) yields $E_{P_0}(X)\leq\text{ess sup}(-X)$ for all $X\in L$. Given $X\in D$, since $\text{ess sup}(-X)\leq 1$ and $X+1\geq 0$ a.s., one obtains
\begin{gather*}
E_{P_0}\abs{X}\leq 1+E_{P_0}(X+1)\leq 2+\text{ess sup}(-X)\leq 3.
\end{gather*}
Hence, $\bigl\{P_0(X\in\cdot):\,X\in D\bigr\}$ is tight. Suppose now that $k_n> 1$ for large $n$, say $k_n> 1$ for each $n\geq m$. Note that $r=\sum_{n=m}^\infty P_0(A_n)\,k_n^{-1}2^{-n}<\infty$ and define
\begin{equation*}
Q(\cdot)=\frac{1}{r}\,\sum_{n=m}^\infty\frac{P_0(\cdot\cap A_n)}{k_n\,2^n}.
\end{equation*}
Then, $Q\in\mathcal{Q}$. Arguing as above, condition (c) implies
\begin{gather*}
E_Q\abs{X}\leq 1+E_Q(X+1)=2+\frac{1}{r}\,\sum_{n=m}^\infty\,\frac{E_{P_0}\bigl\{X\,I_{A_n}\bigr\}}{k_n\,2^n}
\\\leq 2+\frac{1}{r}\,\sum_{n=m}^\infty\frac{\text{ess sup}(-X)}{2^n}\leq 2+\frac{1}{r}\quad\text{for each }X\in D.
\end{gather*}
Thus, $\bigl\{Q(X\in\cdot):\,X\in D\bigr\}$ is tight. Since $Q\sim P_0$, then $\bigl\{P_0(X\in\cdot):\,X\in D\bigr\}$ is tight as well.

{\bf (d) $\Rightarrow$ (b).} Suppose (d) holds. By a result of Yan \cite{YAN}, there is $Q\in\mathcal{Q}$ such that $k:=\sup_{X\in D}E_Q(X)<\infty$, where $D$ is defined as above. Fix $X\in L$ with $P_0(X\neq 0)>0$ and let $Y=X/\text{ess sup}(-X)$. Since $Y\in D$, one obtains
\begin{equation*}
E_Q(X)=E_Q(Y)\,\text{ess sup}(-X)\leq k\,\text{ess sup}(-X).
\end{equation*}

Thus, (b) $\Leftrightarrow$ (c) $\Leftrightarrow$ (d). This concludes the proof of the first part of the theorem, since it is already known that (b) $\Leftrightarrow$ (a) $\Leftrightarrow$ $\mathbb{S}\neq\emptyset$.

Finally, suppose (b) holds for some $Q\in\mathcal{Q}$ and $k\geq 0$. It remains to show that $P=(1+k)^{-1}(Q+k\,P_1)\in\mathbb{S}$ for some $P_1\in\mathbb{P}$. If $k=0$, then $Q\in\mathbb{S}$ and $P=Q$. Thus, suppose $k>0$ and define
\begin{equation*}
C=\{-X:X\in L\},\quad \phi(Z)=-(1/k)\,E_Q(Z)\text{ for }Z\in C,\quad\mathcal{E}=\{A\in\mathcal{A}:P_0(A)=1\}.
\end{equation*}
Given $A\in\mathcal{E}$ and $Z\in C$, since $-Z\in L$ condition (b) yields
\begin{equation*}
\phi(Z)=(1/k)\,E_Q(-Z)\leq\text{ess sup}(Z)\leq\sup_AZ.
\end{equation*}
By Lemma \ref{kumon}, there is $P_1\in\mathbb{P}$ such that $P_1\ll P_0$ and $E_{P_1}(X)\leq -(1/k)\,E_Q(X)$ for all $X\in L$. Since $Q\sim P_0$ and $P_1\ll P_0$, then $P=(1+k)^{-1}(Q+k\,P_1)\sim P_0$. Further,
\begin{equation*}
(1+k)\,E_P(X)=E_Q(X)+k\,E_{P_1}(X)\leq 0\quad\text{for all }X\in L.
\end{equation*}

\end{proof}

Since $L\subset L_\infty$, condition \eqref{nac} can be written as $(L-L_\infty^+)\cap L_\infty^+=\{0\}$. Thus, condition (a) can be seen as a no-arbitrage condition. One more remark is in order. Let $\sigma(L_\infty,L_1)$ denote the topology on $L_\infty$ generated by the maps $Z\mapsto E_{P_0}\bigl(Y\,Z)$ for all $Y\in L_1$. In the early eighties, Kreps and Yan proved that the existence of an ESM amounts to
\begin{equation}\label{g6f}
\overline{L-L_\infty^+}\,\cap L_\infty^+=\{0\}\quad\text{with the closure in }\sigma(L_\infty,L_1);\tag{a*}
\end{equation}
see \cite{K}, \cite{S} and \cite{YAN}. But the geometric meaning of $\sigma(L_\infty,L_1)$ is not so transparent as that of the norm-topology. Hence, a question is what happens if the closure is taken in the norm-topology, that is, if \eqref{g6f} is replaced by (a). The answer, due to \cite[Theorem 2]{CAS} and \cite[Theorem 2.1]{ROK}, is reported in Theorem \ref{t1}.

Note also that, since $L\subset L_\infty$, condition (a) agrees with the {\em no free lunch with vanishing
risk} condition of Delbaen and Schachermayer \cite{DS94}:
\begin{equation*}
\overline{(L-L_0^+)\cap L_\infty}\,\cap L_\infty^+=\{0\}\quad\text{with the closure in the norm-topology}.
\end{equation*}
However, in \cite{DS94}, $L$ is a suitable class of stochastic integrals (in a fixed time interval and driven by a fixed semi-martingale) while $L$ is any convex cone of bounded random variables in Theorem \ref{t1}. Further, the equivalence between $\mathbb{S}\neq\emptyset$ and the no free lunch with vanishing
risk condition is no longer true when $L$ may include unbounded random variables; see Example \ref{nflvr5}.

Let us turn to (b). Once $Q\in\mathcal{Q}$ has been selected, condition (b) provides a simple criterion for $\mathbb{S}\neq\emptyset$. However, choosing $Q$ is not an easy task. The obvious choice is perhaps $Q=P_0$.

\begin{cor}\label{v7t} Let $L$ be a convex cone of real bounded random variables. Condition (b) holds with $Q=P_0$, that is
\begin{equation*}
E_{P_0}(X)\leq k\,\text{ess sup}(-X)\quad\text{for all }X\in L\text{ and some constant }k\geq 0,
\end{equation*}
if and only if there is $P\in\mathbb{S}$ such that $P\geq r\,P_0$ for some constant $r>0$.
\end{cor}

\begin{proof} Let $P\in\mathbb{S}$ be such that $P\geq r\,P_0$. Fix $X\in L$. Since $E_P(X)\leq 0$, then $E_P(X^+)\leq E_P(X^-)$ and $\text{ess sup}(X^-)=\text{ess sup}(-X)$. Hence,
\begin{gather*}
E_{P_0}(X)\leq E_{P_0}(X^+)\leq (1/r)\,E_P(X^+)\leq (1/r)\,E_P(X^-)
\\\leq (1/r)\,\text{ess sup}(X^-)=(1/r)\,\text{ess sup}(-X).
\end{gather*}
Conversely, if condition (b) holds with $Q=P_0$, Theorem \ref{t1} implies that \linebreak $P=(1+k)^{-1}(P_0+kP_1)\in\mathbb{S}$ for suitable $P_1\in\mathbb{P}$. Thus, $P\geq (1+k)^{-1}P_0$.
\end{proof}

Condition (c) is in the spirit of Corollary \ref{v7t} (to avoid the choice of $Q$). It is a sort of localized version of (b), where $Q$ is replaced by a suitable sequence $A_n$ of events. See also \cite[Theorem 5]{BPR2}.

The meaning of condition (d) is quite transparent for those familiar with weak convergence of probability measures. Among other things, (d) depends on $P_0$ only and one of its versions works nicely when $L$ includes unbounded random variables; see Subsection \ref{unb67tyh}. Moreover, condition (d) can be regarded as a no-arbitrage condition. Indeed, basing on \cite[Lemma 2.3]{BS}, it is not hard to see that (d) can be rewritten as:
\begin{itemize}
\item[ ] For each $Z\in L_0^+$, $P_0(Z>0)>0$, there is a constant $a>0$ such that
\begin{equation*}
P_0\bigl(X+1<a\,Z\bigr)>0\quad\text{whenever }X\in L\text{ and }X\geq -1\text{ a.s.}
\end{equation*}
\end{itemize}
Such condition is a market viability condition, called {\em no-arbitrage of the first kind}, investigated by Kardaras in \cite{KARD}-\cite{KARD12}.

\subsection{The unbounded case}\label{unb67tyh}

In dealing with ESFA's, it is crucial that $L\subset L_\infty$. In fact, all arguments (known to us) for existence of ESFA's are based on de Finetti's coherence principle, but the latter works nicely for bounded random variables only. More precisely, the existing notions of coherence for unbounded random variables do not grant a (finitely additive) integral representation; see \cite{BR}  and \cite{BRR01}. On the other hand, $L\subset L_\infty$ is certainly a restrictive assumption. In this Subsection, we try to relax such assumption.

Our strategy for proving $\mathbb{S}\neq\emptyset$ is to exploit condition (d) of Theorem \ref{t1}. To this end, we need a dominance condition on $L$, such as
\begin{equation}\label{d5c}
\text{for each }X\in L,\text{ there is }\lambda>0\text{ such that }\abs{X}\leq\lambda\,Y\,\text{ a.s.}
\end{equation}
where $Y$ is some real random variable. We can (and will) assume $Y\geq 1$.

Condition \eqref{d5c} is less strong than it appears. For instance, it is always true when $L$ is countably generated. In fact, if $L$ is the convex cone generated by a sequence $(X_n:n\geq 1)$ of real random variables, it suffices to let $Y_n=\sum_{i=1}^n\abs{X_i}\,$ in the following lemma.

\begin{lem}
If $Y_1,Y_2,\ldots$ are non negative real random variables satisfying
\begin{equation*}
\text{for each }X\in L,\text{ there are }\lambda>0\text{ and }n\geq 1\text{ such that }\abs{X}\leq\lambda\,Y_n\,\text{ a.s.},
\end{equation*}
then condition \eqref{d5c} holds for some real random variable $Y$.
\end{lem}

\begin{proof}
For each $n\geq 1$, take $a_n>0$ such that $P_0(Y_n>a_n)<2^{-n}$ and define $A=\bigcup_{n=1}^\infty\bigl\{Y_j\leq a_j\text{ for each }j\geq n\bigr\}$. Then,
\begin{equation*}
P_0(A)=1\quad\text{and}\quad Y=1+\sum_{n=1}^\infty\frac{Y_n}{2^na_n}<\infty\,\text{ on }A.
\end{equation*}
Also, condition \eqref{d5c} holds trivially, since $2^na_n\,Y>Y_n$ on $A$ for each $n\geq 1$.
\end{proof}

Next result applies to those convex cones $L$ satisfying condition \eqref{d5c}. It provides a sufficient (sometimes necessary as well) criterion for $\mathbb{S}\neq\emptyset$.

\begin{cor}\label{fol7f3b}
Suppose condition \eqref{d5c} holds for some convex cone $L$ and some random variable $Y$ with values in $[1,\infty)$. Then, $\mathbb{S}\neq\emptyset$ provided
\begin{gather}\label{tight63}
\text{for each }\epsilon>0,\text{ there is }c>0\text{ such that }
\\P_0\bigl(\abs{X}>c\,Y\bigr)<\epsilon\quad\text{whenever }X\in L\text{ and }X\geq -Y\text{ a.s. }\notag
\end{gather}
Conversely, condition \eqref{tight63} holds if $\mathbb{S}\neq\emptyset$ and $Y$ is $P$-integrable for some $P\in\mathbb{S}$.
\end{cor}

\begin{proof} First note that Theorem \ref{t1} is still valid if each member of the convex cone is essentially bounded (even if not bounded). Let $L^*=\{X/Y:X\in L\}$. Then, $L^*$ is a convex cone of essentially bounded random variables and condition \eqref{tight63} is equivalent to tightness of $\bigl\{P_0(Z\in\cdot):\,Z\in L^*,\,Z\geq -1$ a.s.$\bigr\}$. Suppose \eqref{tight63} holds. By Theorem \ref{t1}-(d), $L^*$ admits an ESFA, i.e., there is $T\in\mathbb{P}$ such that $T\sim P_0$ and $E_T(Z)\leq 0$ for all $Z\in L^*$. As noted at the beginning of this Section, such a $T$ can be written as $T=\delta\,P_1+(1-\delta)\,Q$, where $\delta\in [0,1)$, $P_1\in\mathbb{P}$ and $Q\in\mathcal{Q}$. Since $Y\geq 1$,
\begin{equation*}
0<(1-\delta)\,E_Q(1/Y)\leq E_T(1/Y)\leq 1.
\end{equation*}
Accordingly, one can define
\begin{equation*}
P(A)=\frac{E_T\bigl(I_A/Y\bigr)}{E_T(1/Y)}\quad\text{for all }A\in\mathcal{A}.
\end{equation*}
Then, $P\in\mathbb{P}$, $P\sim P_0$, each $X\in L$ is $P$-integrable, and
\begin{equation*}
E_P(X)=\frac{E_T\bigl\{X/Y\bigr\}}{E_T(1/Y)}\leq 0\quad\text{for all }X\in L.
\end{equation*}
Thus, $P\in\mathbb{S}$. Next, suppose $\mathbb{S}\neq\emptyset$ and $Y$ is $P$-integrable for some $P\in\mathbb{S}$. Define
\begin{equation*}
T(A)=\frac{E_P\bigl\{I_A\,Y\bigr\}}{E_P(Y)}\quad\text{for all }A\in\mathcal{A}.
\end{equation*}
Again, one obtains $T\in\mathbb{P}$, $T\sim P_0$ and $E_T(Z)\leq 0$ for all $Z\in L^*$. Therefore, condition \eqref{tight63} follows from Theorem \ref{t1}-(d).
\end{proof}

By Corollary \ref{fol7f3b}, $\mathbb{S}\neq\emptyset$ amounts to condition \eqref{tight63} when $L$ is finite dimensional. In fact, if $L$ is the convex cone generated by the random variables $X_1,\ldots,X_d$, condition \eqref{d5c} holds with $Y=1+\sum_{i=1}^d\abs{X_i}\,$ and such $Y$ is certainly $P$-integrable if $P\in\mathbb{S}$. The case of $L$ finite dimensional, however, is better addressed in the next two Sections; see Theorem \ref{p3b} and Example \ref{i8uh}.

\section{Equivalent super-martingale measures} If suitably strengthened, some of the conditions of Theorem \ref{t1} become equivalent to existence of ESM's. One example is condition (a) (just replace it by \eqref{g6f}). Another example, as we prove in this Section, is condition (b).

Our main result provides a necessary and sufficient condition for ESM's to exist. Such condition looks potentially useful in real problems (at least when applied with $Q=P_0$). Furthermore, unlike Theorem \ref{t1}, $L$ is not requested to consist of bounded random variables.

Recall the notation $\mathcal{Q}=\{Q\in\mathbb{P}_0:Q\sim P_0\}$.

\begin{thm}\label{p3b} Let $L$ be a convex cone of real random variables. There is an ESM if and only if
\begin{gather}\label{bn5r4}
E_Q\abs{X}<\infty\quad\text{and}\quad E_Q(X)\leq k\,E_Q(X^-),\quad X\in L,\tag{b*}
\end{gather}
for some $Q\in\mathcal{Q}$ and some constant $k\geq 0$. Moreover, there is an ESM $P$ such that
\begin{gather*}
r\,P_0\leq P\leq s\,P_0,\quad\text{for some constants }0<r\leq s,
\end{gather*}
if and only if condition \eqref{bn5r4} holds with $Q=P_0$, that is
\begin{gather*}
E_{P_0}\abs{X}<\infty\quad\text{and}\quad E_{P_0}(X)\leq k\,E_{P_0}(X^-)\quad\text{for all }X\in L.
\end{gather*}
\end{thm}

\begin{proof} If there is an ESM, say $P$, condition \eqref{bn5r4} trivially holds with $Q=P$ and any $k\geq 0$. Conversely, suppose \eqref{bn5r4} holds for some $k\geq 0$ and $Q\in\mathcal{Q}$. Define $t=k+1$ and
\begin{gather*}
\mathcal{K}=\bigl\{P\in\mathbb{P}_0:(1/t)\,Q\leq P\leq t\,Q\bigr\}.
\end{gather*}
If $P\in\mathcal{K}$, then $P\in\mathbb{P}_0$, $P\sim Q\sim P_0$ and $E_P\abs{X}\leq t\,E_Q\abs{X}<\infty$ for all $X\in L$. Thus, it suffices to see that $E_P(X)\leq 0$ for some $P\in\mathcal{K}$ and all $X\in L$.

We first prove that, for each $X\in L$, there is $P\in\mathcal{K}$ such that $E_P(X)\leq 0$. Fix $X\in L$ and define $P(A)=E_Q\bigl\{f\,I_A\bigr\}$ for all $A\in\mathcal{A}$, where
\begin{gather*}
f=\frac{I_{\{X\geq 0\}}+t\,I_{\{X< 0\}}}{Q(X\geq 0)+t\,Q(X< 0)}.
\end{gather*}
Since $E_Q(f)=1$ and $(1/t)\leq f\leq t$, then $P\in\mathcal{K}$. Further, condition \eqref{bn5r4} implies
\begin{gather*}
E_P(X)=E_Q\bigl\{f\,X\bigr\}=\frac{E_Q(X^+)-t\,E_Q(X^-)}{Q(X\geq 0)+t\,Q(X< 0)}=\frac{E_Q(X)-k\,E_Q(X^-)}{Q(X\geq 0)+t\,Q(X< 0)}\leq 0.
\end{gather*}

Next, let $\mathcal{Z}$ be the set of all functions from $\mathcal{A}$ into $[0,1]$, equipped with the product topology. Then,
\begin{gather}\label{b2fg3}
\mathcal{K}\text{ is compact and }\,\{P\in\mathcal{K}:E_P(X)\leq 0\}\,\text{ is closed for each }X\in L.
\end{gather}
To prove \eqref{b2fg3}, we fix a net $(P_\alpha)$ of elements of $\mathcal{Z}$ converging to $P\in\mathcal{Z}$, that is, $P_\alpha(A)\rightarrow P(A)$ for each $A\in\mathcal{A}$. If $P_\alpha\in\mathcal{K}$ for each $\alpha$, one obtains $P\in\mathbb{P}$ and $(1/t)\,Q\leq P\leq t\,Q$. Since $Q\in\mathbb{P}_0$ and $P\leq t\,Q$, then $P\in\mathbb{P}_0$, i.e., $P\in\mathcal{K}$. Hence, $\mathcal{K}$ is closed, and since $\mathcal{Z}$ is compact, $\mathcal{K}$ is actually compact. If $X\in L$, $P_\alpha\in\mathcal{K}$ and $E_{P_\alpha}(X)\leq 0$ for each $\alpha$, then $P\in\mathcal{K}$ (for $\mathcal{K}$ is closed). Thus, $E_P\abs{X}<\infty$. Define the set $A_c=\{\abs{X}\leq c\}$ for $c>0$. Since $P_\alpha$ and $P$ are in $\mathcal{K}$, it follows that
\begin{gather*}
\abs{E_{P_\alpha}(X)-E_P(X)}\leq
\\\leq\abs{E_{P_\alpha}\bigl\{X-X\,I_{A_c}\bigr\}}+\abs{E_{P_\alpha}\bigl\{X\,I_{A_c}\bigr\}-E_P\bigl\{X\,I_{A_c}\bigr\}}
+\abs{E_P\bigl\{X\,I_{A_c}-X\bigr\}}
\\\leq E_{P_\alpha}\bigl\{\abs{X}\,I_{\{\abs{X}>c\}}\bigr\}+\abs{E_{P_\alpha}\bigl\{X\,I_{A_c}\bigr\}-E_P\bigl\{X\,I_{A_c}\bigr\}}
+E_P\bigl\{\abs{X}\,I_{\{\abs{X}>c\}}\bigr\}
\\\leq 2\,t\,E_Q\bigl\{\abs{X}\,I_{\{\abs{X}>c\}}\bigr\}+\abs{E_{P_\alpha}\bigl\{X\,I_{A_c}\bigr\}-E_P\bigl\{X\,I_{A_c}\bigr\}}.
\end{gather*}
Since $X\,I_{A_c}$ is bounded, $E_P\bigl\{X\,I_{A_c}\bigr\}=\lim_\alpha E_{P_\alpha}\bigl(X\,I_{A_c}\bigr)$. Thus,
\begin{gather*}
\limsup_\alpha\,\abs{E_{P_\alpha}(X)-E_P(X)}\leq 2\,t\,E_Q\bigl\{\abs{X}\,I_{\{\abs{X}>c\}}\bigr\}\quad\text{for every }c>0.
\end{gather*}
As $c\rightarrow\infty$, one obtains $E_P(X)=\lim_\alpha E_{P_\alpha}(X)\leq 0$. Hence, $\{P\in\mathcal{K}:E_P(X)\leq 0\}$ is closed.

Because of \eqref{b2fg3}, to conclude the proof of the first part it suffices to see that
\begin{gather}\label{bs3e}
\bigl\{P\in\mathcal{K}:E_P(X_1)\leq 0,\ldots,E_P(X_n)\leq 0\bigr\}\neq\emptyset
\end{gather}
for all $n\geq 1$ and $X_1,\ldots,X_n\in L$. Our proof of \eqref{bs3e} is inspired to \cite[Theorem 1]{KEMP}.

Given $n\geq 1$ and $X_1,\ldots,X_n\in L$, define
\begin{gather*}
C=\bigcup_{P\in\mathcal{K}}\bigl\{(a_1,\ldots,a_n)\in\mathbb{R}^n:E_P(X_j)\leq a_j\,\text{ for }j=1,\ldots,n\bigr\}.
\end{gather*}
Then, $C$ is a convex closed subset of $\mathbb{R}^n$. To prove $C$ closed, suppose
\begin{gather*}
(a_1^{(m)},\ldots,a_n^{(m)})\rightarrow (a_1,\ldots,a_n),\,\text{ as }\,m\rightarrow\infty,\,\text{where }\,(a_1^{(m)},\ldots,a_n^{(m)})\in C.
\end{gather*}
For each $m$, take $P_m\in\mathcal{K}$ such that $E_{P_m}(X_j)\leq a_j^{(m)}$ for all $j$. Since $\mathcal{K}$ is compact, $P_\alpha\rightarrow P$ for some $P\in\mathcal{K}$ and some subnet $(P_\alpha)$ of the sequence $(P_m)$. Hence,
\begin{gather*}
a_j=\lim_\alpha a_j^{(\alpha)}\geq\lim_\alpha E_{P_{\alpha}}(X_j)=E_P(X_j)\quad\text{for }j=1,\ldots,n.
\end{gather*}
Thus $(a_1,\ldots,a_n)\in C$, that is, $C$ is closed.

Since $C$ is convex and closed, $C$ is the intersection of all half-planes $\{f\geq u\}$ including it, where $u\in\mathbb{R}$ and $f:\mathbb{R}^n\rightarrow\mathbb{R}$ is a linear functional. Fix $f$ and $u$ such that $C\subset\{f\geq u\}$. Write $f$ as $f(a_1,\ldots,a_n)=\sum_{j=1}^n\lambda_j\,a_j$, where $\lambda_1,\ldots,\lambda_n$ are real coefficients. If $(a_1,\ldots,a_n)\in C$, then $(a_1+b,a_2,\ldots,a_n)\in C$ for $b>0$, so that
\begin{gather*}
b\,\lambda_1+f(a_1,\ldots,a_n)=f(a_1+b,a_2,\ldots,a_n)\geq u\quad\text{for all }b>0.
\end{gather*}
Hence, $\lambda_1\geq 0$. By the same argument, $\lambda_j\geq 0$ for all $j$, and this implies $f(X_1,\ldots,X_n)\in L$. Take $P\in\mathcal{K}$ such that $E_P\bigl\{f(X_1,\ldots,X_n)\bigr\}\leq 0$. Since $\bigl(E_P(X_1),\ldots,E_P(X_n)\bigr)\in C\subset\{f\geq u\}$, it follows that
\begin{gather*}
u\leq f\Bigl(\bigl(E_P(X_1),\ldots,E_P(X_n)\bigr)\Bigr)=E_P\bigl\{f(X_1,\ldots,X_n)\bigr\}\leq 0=f(0,\ldots,0).
\end{gather*}
This proves $(0,\ldots,0)\in C$ and concludes the proof of the first part.

We finally turn to the second part of the theorem. If condition \eqref{bn5r4} holds with $Q=P_0$, what already proved implies the existence of an ESM $P$ such that $(1/t)\,P_0\leq P\leq t\,P_0$. Conversely, let $P$ be an ESM satisfying $r\,P_0\leq P\leq s\,P_0$ for some $0<r\leq s$. Then, for each $X\in L$, one obtains $E_{P_0}\abs{X}\leq (1/r)\,E_P\abs{X}<\infty$ and
\begin{gather*}
E_{P_0}(X)\leq E_{P_0}(X^+)\leq (1/r)\,E_P(X^+)\leq (1/r)\,E_P(X^-)\leq (s/r)\,E_{P_0}(X^-).
\end{gather*}

\end{proof}

If $L$ is a linear space, condition \eqref{bn5r4} can be written in some other ways. One of these ways is
\begin{gather}\label{bnhkjp5r4}
E_Q\abs{X}<\infty\quad\text{and}\quad \abs{\,E_Q(X)\,}\leq c\,E_Q\abs{X}\tag{b**}
\end{gather}
for all $X\in L$, some $Q\in\mathcal{Q}$ and some constant $c<1$. In fact, \eqref{bn5r4} implies \eqref{bnhkjp5r4} with $c=k/(k+2)$ while \eqref{bnhkjp5r4} implies \eqref{bn5r4} with $k=2c/(1-c)$. In the sequel, when $L$ is a linear space, we make often use of condition \eqref{bnhkjp5r4}. However, we note that \eqref{bnhkjp5r4} is stronger than \eqref{bn5r4} if $L$ fails to be a linear space. For instance, \eqref{bn5r4} holds and \eqref{bnhkjp5r4} fails for the convex cone $L=\{X_b:b\leq 0\}$, where $X_b(\omega)=b$ for all $\omega\in\Omega$.

A last remark is that, if condition \eqref{d5c} holds for some $Y$, then $E_Q\abs{X}<\infty$ for all $X\in L$ can be replaced by $E_Q(Y)<\infty$ in both conditions \eqref{bn5r4} and \eqref{bnhkjp5r4}.

\section{Examples}\label{sg3}

In this Section, $L$ is a linear space. Up to minor changes, however, most examples could be adapted to a convex cone $L$. Recall that, since $L$ is a linear space, \linebreak $E_P(X)=0$ whenever $X\in L$ and $P$ is an ESFA or an ESM.

\begin{ex}\label{i8uh} {\bf (Finite dimensional spaces).} Let $X_1,\ldots,X_d$ be real random variables on $(\Omega,\mathcal{A},P_0)$. Is there a $\sigma$-additive probability $P\in\mathbb{P}_0$ such that
\begin{equation*}
P\sim P_0,\quad E_P\abs{X_j}<\infty\quad\text{and}\quad E_P(X_j)=0\quad\text{for all }j\,\,\,?
\end{equation*}
The question looks natural and the answer is intuitive as well. Such a $P$ exists if and only if $L\cap L_0^+=\{0\}$, that is \eqref{nac} holds, with
\begin{equation*}
L=\,\text{linear space generated by }X_1,\ldots,X_d.
\end{equation*}
This is a known result (for instance, it follows from \cite[Theorem 2.4]{DMW}). However, to our knowledge, such result does not admit elementary proofs. We now deduce it as an immediate consequence of Theorem \ref{p3b}.

Up to replacing $X_j$ with $Y_j=\frac{X_j}{1+\sum_{i=1}^d\abs{X_i}}$, it can be assumed $E_{P_0}\abs{X_j}<\infty$ for all $j$. Let $K=\{X\in L:E_{P_0}\abs{X}=1\}$, equipped with the $L_1$-norm. If $L\cap L_0^+=\{0\}$, then $\abs{\,E_{P_0}(X)\,}<1$ for each $X\in K$. Since $K$ is compact and $X\mapsto E_{P_0}(X)$ is continuous, $\sup_{X\in K}\abs{\,E_{P_0}(X)\,}<1$. Thus, condition \eqref{bnhkjp5r4} holds with $Q=P_0$.

Two remarks are in order. First, if $E_{P_0}\abs{X_j}<\infty$ for all $j$ (so that the $X_j$ should not be replaced by the $Y_j$) the above argument implies that $P$ can be taken to satisfy $r\,P_0\leq P\leq s\,P_0$ for some $0<r\leq s$. Second, Theorem \ref{p3b} also yields a reasonably simple proof of \cite[Theorem 2.6]{DMW}, i.e., the main result of \cite{DMW}.

\end{ex}

\begin{ex}\label{cd5t} {\bf (A question by Rokhlin and Schachermayer).} Suppose that $E_{P_0}(X_n)=0$ for all $n\geq 1$, where the $X_n$ are real bounded random variables. Let $L$ be the linear space generated by the sequence $(X_n:n\geq 1)$ and
\begin{equation*}
P_f(A)=E_{P_0}\bigl\{f\,I_A\bigr\},\quad A\in\mathcal{A},
\end{equation*}
where $f$ is a strictly positive measurable function on $\Omega$ such that $E_{P_0}(f)=1$. Choosing $P_0$, $f$ and $X_n$ suitably, in \cite[Example 3]{ROKSCA} it is shown that

\begin{itemize}

\item[(i)] There is a bounded finitely additive measure $T$ on $\mathcal{A}$ such that
\begin{equation*}
T\ll P_0,\quad T(A)\geq P_f(A)\quad\text{and}\quad\int X\,dT=0\quad\text{for all }A\in\mathcal{A}\text{ and }X\in L;
\end{equation*}

\item[(ii)] {\it No} measurable function $g:\Omega\rightarrow [0,\infty)$ satisfies
\begin{equation*}
g\geq f\text{ a.s.,}\quad E_{P_0}(g)<\infty\quad\text{and}\quad E_{P_0}\bigl\{g\,X\bigr\}=0\text{ for all }X\in L.
\end{equation*}

\end{itemize}

In \cite[Example 3]{ROKSCA}, $L$ is spanned by a (infinite) sequence. Thus, at page 823, the question is raised of whether (i)-(ii) can be realized when $L$ is finite dimensional.

We claim that the answer is no. Suppose in fact that $L$ is generated by the bounded random variables $X_1,\ldots,X_d$. Since $P_f\sim P_0$ and $E_{P_0}(X)=0$ for all $X\in L$, then $L\cap L_0^+=\{0\}$ under $P_f$ as well. Arguing as in Example \ref{i8uh}, one obtains $E_Q(X)=0$, $X\in L$, for some $Q\in\mathbb{P}_0$ such that $r\,P_f\leq Q\leq s\,P_f$, where $0<r\leq s$. Therefore, a function $g$ satisfying the conditions listed in (ii) is $g=\psi/r$, where $\psi$ is a density of $Q$ with respect to $P_0$.

\end{ex}

\begin{ex} {\bf (Example 7 of \cite{BPR2} revisited).} Let $L$ be the linear space generated by the random variables $X_1,X_2,\ldots$, where each $X_n$ takes values in $\{-1,1\}$ and
\begin{equation}\label{b7u8}
P_0\bigl(X_1=x_1,\ldots,X_n=x_n\bigr)>0\quad\text{for all }n\geq 1\text{ and }x_1,\ldots,x_n\in\{-1,1\}.
\end{equation}
Every $X\in L$ can be written as $X=\sum_{j=1}^nb_jX_j$ for some $n\geq 1$ and $b_1,\ldots,b_n\in\mathbb{R}$. By \eqref{b7u8},
\begin{equation*}
\text{ess sup}(X)=\abs{b_1}+\ldots+\abs{b_n}=\text{ess sup}(-X).
\end{equation*}
Hence, condition (b) is trivially true, and Theorem \ref{t1} implies the existence of an ESFA. However, ESM's can fail to exist. To see this, let $P_0(X_n=-1)=(n+1)^{-2}$ and fix $Q\in\mathcal{Q}$. Under $P_0$, the Borel-Cantelli lemma yields $X_n\overset{a.s.}\longrightarrow 1$. Hence, $X_n\overset{a.s.}\longrightarrow 1$ under $Q$ as well, and $Q$ fails to be an ESM for $E_Q(X_n)\rightarrow 1$. This is basically Example 7 of \cite{BPR2}. We now make two remarks on such example.

First, points (i)-(ii) of previous Example \ref{cd5t} hold true for {\it every} strictly positive $f\in L_1$. Fix in fact a measurable function $f:\Omega\rightarrow (0,\infty)$ with $E_{P_0}(f)=1$. Then,
\begin{equation*}
E_{P_f}(X)\leq\text{ess sup}(X)=\text{ess sup}(-X)\quad\text{for all }X\in L.
\end{equation*}
By Corollary \ref{v7t}, there are $r>0$ and $P\in\mathbb{S}$ such that $P\geq r\,P_f$. Hence, point (i) is satisfied by $T=P/r$. If $g$ meets the conditions listed in (ii), then
\begin{equation*}
Q(A)=\frac{E_{P_0}\bigl\{g\,I_A\bigr\}}{E_{P_0}(g)},\quad A\in\mathcal{A},
\end{equation*}
is an ESM. Therefore, point (ii) holds true as well.

Second, Example 7 of \cite{BPR2} can be modified, preserving the possible economic meaning (provided the $X_n$ are regarded as asset prices) but allowing for ESM's to exist. Let $N$ be a random variable, independent of the sequence $(X_n)$, with values in $\{1,2,\ldots\}$. To fix ideas, suppose $P_0(N=n)>0$ for all $n\geq 1$. Take $L$ to be the collection of $X$ of the type
\begin{equation*}
X=\sum_{j=1}^Nb_jX_j
\end{equation*}
for all real sequences $(b_j)$ such that $\sum_j\abs{b_j}<\infty$. Then, $L$ is a linear space of bounded random variables. Given $n\geq 1$, define $L_n$ to be the linear space spanned by $X_1,\ldots,X_n$. Because of \eqref{b7u8} and the independence between $N$ and $(X_n)$, for each $X\in L_n$ one obtains
\begin{equation*}
P_0\bigl(X>0\mid N=n\bigr)>0\quad\Longleftrightarrow\quad P_0\bigl(X<0\mid N=n\bigr)>0.
\end{equation*}
Hence, condition \eqref{nac} holds with $P_0\bigl(\cdot\mid N=n\bigr)$ and $L_n$ in the place of $P_0$ and $L$. Arguing as in Example \ref{i8uh}, it follows that $E_{Q_n}(X)=0$ for all $X\in L_n$ and some $Q_n\in\mathbb{P}_0$ such that $Q_n\sim P_0\bigl(\cdot\mid N=n\bigr)$. Since $Q_n(N=n)=1$, then $E_{Q_n}(X)=0$ for all $X\in L$. Thus, an ESM is $Q=\sum_{n=1}^\infty2^{-n}Q_n$.

Incidentally, in addition to be an ESM for $L$, such a $Q$ also satisfies
\begin{equation*}
E_Q\Bigl(\,\sum_{j=1}^{N\wedge n}b_jX_j\Bigr)=0\quad\text{for all }n\geq 1\text{ and }b_1,\ldots,b_n\in\mathbb{R}.
\end{equation*}
\end{ex}

\begin{ex}\label{s4r} {\bf (Approximating ESM's via ESFA's).} Suppose $L$ consists of bounded random variables and, for each $\epsilon>0$, there is $Q_\epsilon\in\mathcal{Q}$ such that
\begin{equation*}
E_{Q_\epsilon}(X)\leq\epsilon\,\text{ess sup}(-X)\quad\text{for all }X\in L.
\end{equation*}
In view of Theorem \ref{t1},
\begin{equation*}
P_\epsilon=\frac{Q_\epsilon+\epsilon\,T_\epsilon}{1+\epsilon}\in\mathbb{S}\quad\text{for some }T_\epsilon\in\mathbb{P}.
\end{equation*}
Thus, for each $\epsilon>0$, there is an ESFA $P_\epsilon$ whose $\sigma$-additive equivalent part $Q_\epsilon$ has weight $(1+\epsilon)^{-1}$. Nevertheless, as shown in \cite[Example 9]{BPR2}, ESM's can fail to exist. We now give an example more effective than \cite[Example 9]{BPR2}.

Let $Y$ and $Z$ be random variables which, under $P_0$, are i.i.d. with a Poisson distribution with parameter 1. Take $L$ to be the linear space generated by the sequence $(X_j:j\geq 0)$, where
\begin{equation*}
X_0=I_{\{Y=0\}}-I_{\{Z=0\}}\quad\text{and}\quad X_j=I_{\{Y=j\}}-P_0(Y=j)\,I_{\{Z>0\}}\quad\text{for }j>0.
\end{equation*}
If $P\in\mathbb{P}$ meets $P(Z>0)>0$ and $E_P(X_j)=0$ for each $j\geq 0$, then
\begin{gather*}
\sum_{j=0}^\infty P(Y=j)=P(Z=0)+P(Z>0)\,\sum_{j=1}^\infty P_0(Y=j)
\\=P(Z=0)+P(Z>0)\,P_0(Y>0)<1.
\end{gather*}
Hence, no ESM is available. However, given $\epsilon>0$, one can define
\begin{gather*}
Q_\epsilon(\cdot)=\frac{\epsilon\,P_0(\cdot\mid B)+P_0(\cdot\mid B^c)}{\epsilon+1}\quad\text{where }B=\{Y>0\}\cup\{Z>0\}.
\end{gather*}
Fix $X\in L$, say $X=\sum_{j=0}^nb_jX_j$ where $n\geq 1$ and $b_0,\ldots,b_n\in\mathbb{R}$, and define \linebreak $b=\sum_{j=1}^nb_jP_0(Y=j)$. Since $-X=b$ on the set $\{Y>n,Z>0\}$, then $\text{ess sup}(-X)\geq b$. Since $E_{P_0(\cdot\mid B)}(X_0)=0$ and $X_j=0$ on $B^c$ for all $j\geq 0$, one obtains
\begin{gather*}
E_{Q_\epsilon}(X)=\frac{\epsilon}{\epsilon+1}\,\sum_{j=1}^nb_jE_{P_0(\cdot\mid B)}(X_j)=\frac{b\,\epsilon}{\epsilon+1}\,\frac{P_0(Z=0)}{P_0(B)}.
\end{gather*}
Hence, $E_{Q_\epsilon}(X)\leq\epsilon\,\text{ess sup}(-X)$ follows from
\begin{gather*}
\frac{P_0(Z=0)}{P_0(B)}=\frac{P_0(Z=0)}{1-P_0(Z=0)^2}=\frac{e^{-1}}{1-e^{-2}}< 1.
\end{gather*}

\end{ex}

\begin{ex}\label{nflvr5} {\bf (No free lunch with vanishing risk).} It is not hard to see that $\mathbb{S}\neq\emptyset$ implies
\begin{equation*}
\overline{(L-L_0^+)\cap L_\infty}\,\cap L_\infty^+=\{0\}\quad\text{with the closure in the norm-topology of }L_\infty.
\end{equation*}
Unlike the bounded case (see the remarks after Theorem \ref{t1}), however, the converse is not true.

Let $Z$ be a random variable such that $Z>0$ and $P_0(a<Z<b)>0$ for all $0\leq a<b$. Take $L$ to be the linear space generated by $(X_n:n\geq 0)$, where
\begin{gather*}
X_0=Z\,\sum_{k\geq 0}(-1)^kI_{\{k\leq Z<k+1\}}\quad\text{and}
\\X_n=I_{\{Z<n\}}+Z\,\sum_{k\geq n}(-1)^kI_{\{k+2^{-n}\leq Z<k+1\}}\quad\text{for }n\geq 1.
\end{gather*}
Also, fix $P\in\mathbb{P}$ such that $X_n$ is $P$-integrable for each $n\geq 0$ and $P=\delta\,P_1+(1-\delta)\,Q$ for some $\delta\in [0,1)$, $P_1\in\mathbb{P}$ and $Q\in\mathcal{Q}$. From the definition of $P$-integrability (recalled in Section \ref{not6}) one obtains
\begin{gather*}
E_P(X_n)=P(Z<n)+\sum_{k\geq n}(-1)^kE_P\Bigl\{Z\,I_{\{k+2^{-n}\leq Z<k+1\}}\Bigl\}\quad\text{for }n\geq 1.
\end{gather*}
Since $Z=\abs{X_0}$ is $P$-integrable, then
\begin{gather*}
\Abs{\,\sum_{k\geq n}(-1)^kE_P\Bigl\{Z\,I_{\{k+2^{-n}\leq Z<k+1\}}\Bigl\}}\leq \sum_{k\geq n}E_P\Bigl\{Z\,I_{\{k\leq Z<k+1\}}\Bigl\}
\\= E_P\bigl\{Z\,I_{\{Z\geq n\}}\bigr\}\longrightarrow 0\quad\text{as }n\rightarrow\infty.
\end{gather*}
It follows that
\begin{gather*}
\liminf_nE_P(X_n)=\liminf_nP(Z<n)\geq (1-\delta)\,\liminf_nQ(Z<n)=(1-\delta)>0.
\end{gather*}
Hence $P\notin\mathbb{S}$, and this implies $\mathbb{S}=\emptyset$ since each member of $\mathbb{S}$ should satisfy the requirements asked to $P$. On the other hand, it is easily seen that
\begin{equation*}
\text{ess sup}(X)=\infty\quad\text{for each }X\in L\text{ such that }P_0(X\neq 0)>0.
\end{equation*}
Thus, $(L-L_0^+)\cap L_\infty=-L_\infty^+$ which trivially implies
\begin{equation*}
\overline{(L-L_0^+)\cap L_\infty}\,\cap L_\infty^+=\overline{(-L_\infty^+)}\,\cap L_\infty^+=(-L_\infty^+)\cap L_\infty^+=\{0\}.
\end{equation*}
\end{ex}

\begin{ex}\label{u6} {\bf (Equivalent probability measures with given marginals).} Let
\begin{equation*}
\Omega=\Omega_1\times\Omega_2\quad\text{and}\quad\mathcal{A}=\mathcal{A}_1\otimes\mathcal{A}_2
\end{equation*}
where $(\Omega_1,\mathcal{A}_1)$ and $(\Omega_2,\mathcal{A}_2)$ are measurable spaces. Fix a ($\sigma$-additive) probability $T_i$ on $\mathcal{A}_i$ for $i=1,2$. Is there a $\sigma$-additive probability $P\in\mathbb{P}_0$ such that
\begin{equation}\label{stanc}
P\sim P_0\quad\text{and}\text\quad P\bigl(\cdot\times\Omega_2\bigr)=T_1(\cdot),\,\,P\bigl(\Omega_1\times\cdot\bigr)=T_2(\cdot)\,\,?
\end{equation}

Again, the question looks natural (to us). Nevertheless, as far as we know, such a question has been neglected so far. For instance, the well known results by Strassen \cite{STR} do not apply here, for $\mathcal{Q}$ fails to be closed in any reasonable topology on $\mathbb{P}_0$. However, a possible answer can be manufactured via Theorem \ref{p3b}.

Let $L_i$ be a class of bounded measurable functions on $\Omega_i$, $i=1,2$. Suppose each $L_i$ is both a linear space and a determining class, in the sense that, if $R$ and $T$ are ($\sigma$-additive) probabilities on $\mathcal{A}_i$ then
\begin{equation*}
R=T\quad\Longleftrightarrow\quad E_R(f)=E_T(f)\text{ for all }f\in L_i.
\end{equation*}
Define $L$ to be the class of random variables $X$ on $\Omega=\Omega_1\times\Omega_2$ of the type
\begin{equation*}
X(\omega_1,\omega_2)=\bigl\{f(\omega_1)-E_{T_1}(f)\bigr\}+\bigl\{g(\omega_2)-E_{T_2}(g)\bigr\}
\end{equation*}
for all $f\in L_1$ and $g\in L_2$. Then, $L$ is a linear space of bounded random variables. Furthermore, there is $P\in\mathbb{P}_0$ satisfying \eqref{stanc} if and only if $L$ admits an ESM. In turn, by Theorem \ref{p3b}, the latter fact amounts to
\begin{equation*}
\inf_{Q\in\mathcal{Q}}\,\sup_{X\in L\setminus\{0\}}\,\frac{\abs{\,E_Q(X)\,}}{E_Q\abs{X}}<1.
\end{equation*}
Here, condition \eqref{bn5r4} has been replaced by condition \eqref{bnhkjp5r4} since $L$ is a linear space, and $E_Q\abs{X}<\infty$ is because each $X\in L$ is bounded. Further, $X\in L\setminus\{0\}$ stands for $X\in L$ and $P_0(X\neq 0)>0$.

So far, we tacitly assumed that only the $\sigma$-additive solutions of \eqref{stanc} make some interest. But this is not necessarily true, and one could be interested in a finitely additive solution as well. Then, it suffices to apply Theorem \ref{t1}. For every $i=1,2$, take $L_i$ to be the collection of all simple functions with respect to $(\Omega_i,\mathcal{A}_i)$. Basing on (say) condition (b), there is $P\in\mathbb{P}$ satisfying \eqref{stanc} if and only if
\begin{gather*}
\text{ess sup}(X)>0\quad\text{for all }X\in L\setminus\{0\}\quad\text{and}
\\\inf_{Q\in\mathcal{Q}}\,\sup_{X\in L\setminus\{0\}}\,\frac{E_Q(X)}{\text{ess sup}(-X)}<\infty.
\end{gather*}

\end{ex}

\end{document}